\documentclass[11pt,oneside]{article}

\usepackage{color}
\usepackage{graphicx} 
\usepackage[centertags]{amsmath}
\usepackage{hyperref}
\usepackage{amssymb}
\usepackage{amsthm}
\usepackage[english]{babel}
\usepackage[noblocks,affil-sl]{authblk}             
\usepackage{setspace}             
\usepackage{fullpage}            
\newtheorem{thm}{Theorem}[section]
\newtheorem{cor}[thm]{Corollary}

\newtheorem{prop}[thm]{Proposition}
\newtheorem{defn}[thm]{Definition}
\newtheorem{rem}[thm]{Remark}

\numberwithin{equation}{section}
\newcommand{\Real}{\mathbb R}

\newcommand{\vareps}{\epsilon}
\newcommand{\To}{\longrightarrow}
\newcommand{\A}{\mathcal{A}}
\newcommand{\punt}{\mathbf{.}}
\begin{document}
\title{On some applications of a symbolic representation of non-centered L\'evy processes}
\author{E. Di Nardo} \affil{Department of Mathematics, Computer Science and Economics, Universit\`a  
della Basilicata, Potenza, Italy}
\author{I. Oliva} \affil{Department of Economics Universit\`a di Verona, Bologna, Italy}
%
\date{\today}
\maketitle
\begin{abstract}
By using a symbolic technique known in the literature as the classical umbral calculus, we characterize two classes 
of polynomials related to L\'evy processes: the Kailath-Segall and the time-space harmonic polynomials. 
We provide the Kailath-Segall formula in terms of cumulants and we recover simple 
closed-forms for several families of polynomials with respect to 
not centered L\'evy processes, such as the Hermite polynomials with the Brownian motion, the Poisson-Charlier 
polynomials with the Poisson processes, the actuarial polynomials with the Gamma processes, the first kind  Meixner polynomials with the Pascal processes, the Bernoulli, Euler and Krawtchuk polynomials with suitable 
random walks.
\end{abstract}
\textsf{\textbf{keywords}: L\'evy process, time-space harmonic polynomial, Kailath-Segall polynomial, cumulant, 
umbral calculus}
%
\section{Introduction}
The umbral calculus is a symbolic method, known in the literature since the XIX century, consisting in a set of 
mathematical tricks, dealing with number sequences, whose subscripts were treated as they were powers. No formal 
setting for this theory was given until 1964, when Gian-Carlo Rota disclosed the \lq\lq umbral magic art\rq\rq 
of lowering and raising exponents, bringing to the light the underlying linear functional \cite{Rota}. From 
1964 on, the umbral calculus was deeply developed. In particular, in 1994, Rota and Taylor \cite{SIAM} provided a 
simple presentation of the umbral calculus in a framework very similar to the theory of random variables and, in 2001, 
Di Nardo and Senato \cite{Dinsen} gave a complete formalization of the matter.

Here, we refer to the classical umbral calculus as a syntax consisting in an alphabet $\A = \{\alpha, \beta, \gamma, 
\ldots\}$ of symbols, called \emph{umbrae}, and a suitable linear functional $E,$ called \emph{evaluation}, which 
resembles the expectation operator in probability theory. Therefore umbrae look like the framework of random variables, 
with no reference to any probability space. The key point of the theory is the idea of associating a unital number sequence 
$1, a_1, a_2, \ldots$ to a sequence $1, \alpha, \alpha^2, \ldots$ of powers of $\alpha$ by means of the evaluation functional. 

In this framework, the notion of summation of umbrae can be extended to the case of a non-integer number of 
addends, thus leading us to a symbolic version of the infinite divisibility property and therefore of 
L\'evy processes \cite{Sato}. 

In 1997, together with Wallstrom \cite{Wall}, Rota conceived a combinatorial definition of stochastic integration in the 
setting of random measures. The starting point is the Kailath-Segall formula \cite{KS} interpreted 
in combinatorial terms and applied to derive recursion relations for some classes of orthogonal polynomials.
The Kailath-Segall formula links the variations $\{X_t^{(n)}\}_{t \geq 0}$ of a L\'evy process
\begin{equation}
X_t^{(1)} = X_t, \,\,\, X_t^{(2)} = [X,X]_t, \,\,\, X_t^{(n)} = \sum_{s \geq t}(\Delta X_s)^n, 
\;\; n \geq 3,
\label{prima}
\end{equation}
to its iterated stochastic integrals
\begin{equation} 
P_t^{(0)} = 1, \,\,\, P_t^{(1)}=X_t, \,\,\, P_t^{(n)}=\int_0^t P_{s-}^{(n-1)} {\rm d}X_s, \;\; n \geq 2
\label{seconda}
\end{equation}
by using suitable polynomials, named the {\it Kailath-Segall polynomials}. In this paper, we give an umbral representation 
of this class of polynomials highlighting the role played by their cumulants. We show that the Kailath-Segall formula 
is a suitable generalization of the well-known formulae giving elementary symmetric polynomials in terms of
power sum symmetric polynomials. 

Cumulants play the same role in the umbral expression of time-space harmonic polynomials with respect 
to not necessarily centered L\'evy processes. A family of polynomials $\{P(x,t)\}_{t \geq 0}$ is said to be time-space harmonic with
respect to a  L\'evy process $\{X_t\}_{t \geq 0}$ if $E[P(X_t,t) \,\, |\; \mathcal{F}_{v}]
=P(X_v,v),$ for all $v \leq t,$ where $\mathcal{F}_{v}=\sigma\left( X_\tau : \tau \leq v\right)$
is the natural filtration associated with $\{X_t\}_{t \geq 0}.$ A L\'evy process is not necessarily 
a martingale. Therefore to find polynomials such that it is a martingale the stochastic process obtained by replacing the indeterminate 
$x$ with the L\'evy process $\{X_t\}_{t \geq 0},$ becomes fundamental, especially for applications in 
mathematical finance \cite{CKT}. In \cite{Sole}, to get a characterization of time-space harmonic polynomials,
the authors use the Teugels martingale and refer to centered L\'evy processes for which the martingale property holds. 
In this paper, we focus our attention on non-centered L\'evy processes, which do not share the martingale property
and we show how the classical umbral calculus allows us to get more general results without taking advantage
of the martingale property. Moreover, the umbral expression of these polynomials relies on a very simple 
closed-form of the corresponding coefficients which can be easily implemented in any symbolic software, see \cite{Dinoli} as example. 

The paper is structured as follows. Section 2 is provided for readers unaware of the classical umbral calculus. We 
have chosen to recall terminology, notation and the basic definitions strictly necessary to deal with the object of 
this paper. We skip any proof. The reader interested in is referred to \cite{Dinsen, Dinardoeurop}. Section 3 gives 
the umbral expression of L\'evy processes and analyses the classes of Kailath-Segall and time-space harmonic polynomials. 
In Section 4, we give umbral expressions of many classical families of polynomials as time-space harmonic 
with respect to suitable L\'evy processes.
%
\section{Background on the classical umbral calculus}
The classical umbral calculus is a syntax consisting of the following data: 

{\it (i)} a set $\A =\{\alpha, \beta, \gamma, \ldots\}$ of objects, called \emph{umbrae};

{\it (ii)} an \emph{evaluation} linear functional $E\,:\, \Real[x][\A]\To \Real[x]$, where $\Real$ is the 
field of real numbers, such that $E[1]=1$ and the \emph{uncorrelation property} holds
$$E[x^n \alpha^i\beta^j\gamma^k\cdots]=x^n E[\alpha^i]E[\beta^j]E[\gamma^k]\cdots$$
for all $\alpha, \beta, \gamma, \ldots \in \A$ and for all nonnegative integers $n, i, j, k, ,\ldots$ 

{\it (iii)} the \emph{augmentation} umbra $\vareps \in \A,$ with $E[\vareps^n] = \delta_{0,n},$ for all 
nonnegative integers $n,$ where $\delta_{0,n}$ is the Kronecker symbol, that is, $\delta_{0,n}$ is equal to $1$ if 
$n = 0$ and $0$ otherwise;

{\it (iv)} the \emph{unity umbra} $u \in \A,$ with $E[u^n] = 1,$ for all nonnegative integers $n.$

A sequence $a_0 = 1, a_1, a_2, \ldots \in \Real[x]$ is \emph{umbrally represented} by an umbra
$\alpha$ if $E[\alpha^n] = a_n,$ for all $n \geq 0.$ The element $a_n$ is the $n$-th \emph{moment} of
the umbra $\alpha.$ An umbra is said to be \emph{scalar} (respectively, \emph{polynomial}) if its
moments are in $\Real$ (respectively, in $\Real[x]$). A polynomial in $\Real[\A]$ is an \emph{umbral 
polynomial}. The \emph{generating function} of an umbra $\alpha \in \A$ is the formal power series $f(\alpha, z) 
= 1 + \sum_{n \geq 1} a_n z^n/ n!,$ for which we do not take into account any question of convergence \cite{Stanley}.

Special umbrae are the \emph{singleton umbra} $\chi \in \A,$ with $f(\chi,z) = 1 + z;$ the \emph{boolean unity} 
$\bar{u} \in \A,$ with $f(\bar{u},z) = 1/(1-z);$ the \emph{Bell umbra} $\beta \in \A,$ with $f(\beta,z) = \exp(e^z - 1)$
and moments the Bell numbers; the \emph{Bernoulli umbra} $\iota,$ with $f(\iota,z)=z/(e^z -1)$ and moments the 
Bernoulli numbers; the \emph{Euler umbra} $\eta,$ with $f(\eta, z) = 2 e^z/(1 + e^{2 z})$ and moments the Euler numbers. 

The alphabet $\A$ can be extended with new symbols arising from operations among
umbrae. These new umbrae are called \emph{auxiliary umbrae} and the resulting umbral calculus
is said to be \emph{saturated} \cite{SIAM}. Some useful auxiliary umbrae are recalled
in the following. 

{\bf Disjoint sum and difference.} Given $\alpha, \gamma \in \A,$ their
\emph{disjoint sum} $\alpha \, \dot{+} \, \gamma$ (respectively, \emph{disjoint difference}
$\alpha \, \dot{-} \, \gamma$) is such that
$f(\alpha \, \dot{+} \, \gamma,z) = f(\alpha,z) + f(\gamma,z) - 1$
(respectively, $f(\alpha \, \dot{-} \, \gamma,z) = f(\alpha,z) - f(\gamma,z) + 1$).

{\bf Dot-product.} First, let us observe that there are infinitely many and distinct umbrae
representing the same sequence of moments. More precisely, the umbrae $\alpha$ and $\gamma$
are said to be \emph{similar} if $E[\alpha^n]=E[\gamma^n]$ for all nonnegative integers $n,$ 
in symbols $\alpha \equiv \gamma.$ Now let us consider $n$ uncorrelated umbrae $\alpha',\alpha'',\ldots,\alpha'''$
similar to $\alpha$ and take their summation: the resulting umbra $\alpha' + \alpha'' + \cdots
+ \alpha'''$ is denoted by the symbol $n \punt \alpha.$ The umbra $n \punt \alpha$ is called
the \emph{dot-product} of the integer $n$ and the umbra $\alpha.$ Its generating function is $f(n \punt \alpha,z) =
(f(\alpha,z))^n$ and the moments are $E[(n \punt \alpha)^i] = \sum_{j=1}^i (n)_j B_{i,j}(a_1,\ldots,a_{i-j+1}),$
where $(n)_j$ is the lower factorial and $B_{i,j}$ are the partial exponential
Bell polynomials \cite{Dinardoeurop}. The integer $n$ can be replaced by any $t \in \Real$ so that
\begin{equation}
E[(t \punt \alpha)^i] = \sum_{j=1}^i (t)_j B_{i,j}(a_1,\ldots,a_{i-j+1}).
\label{(momlevy)}
\end{equation}
In particular, we have 
\begin{equation}
t \punt (\alpha + \gamma) \equiv t \punt \alpha + t \punt \gamma.
\label{(dist)}
\end{equation}
If $t=-1$ the umbra $-1 \punt \alpha$ is called the inverse of $\alpha.$ We have $-1 \punt \alpha + \alpha \equiv 
\vareps.$ In (\ref{(momlevy)}), we can also replace $t$ by any umbra $\gamma \in \A,$ for more details see \cite{Dinardoeurop}. 
If $(\gamma)_j= \gamma (\gamma - 1) \cdots (\gamma -j + 1)$ denotes the lower factorial polynomial, then we have
$E[(\gamma \punt \alpha)^i] = \sum_{j=1}^i E[(\gamma)_j] B_{i,j}(a_1,\ldots,a_{i-j+1}).$ 
The umbra $\gamma \punt \alpha$ is called the \emph{dot-product} of the umbrae $\alpha$
and $\gamma.$ Special dot-product umbrae are $\chi \punt \alpha$ and $\beta \punt \alpha.$
The umbra $\chi \punt \alpha$ is denoted by the symbol $\kappa_{{\scriptscriptstyle \alpha}}$ and called the 
$\alpha$-\emph{cumulant umbra} \cite{Bernoulli}, since $f(\kappa_{{\scriptscriptstyle \alpha}}, z) = 1 + \log(f(\alpha, z)).$
The umbra $\beta \punt \alpha$ is called the $\alpha$-\emph{partition umbra}.
In particular, we have $\alpha \equiv \beta \punt \kappa_{{\scriptscriptstyle \alpha}}$
and $\beta \punt \chi \equiv \chi \punt \beta \equiv u.$ Later on, we will often 
use the following properties for the cumulant umbra and the partition umbra: 
\begin{equation}
\chi \punt (\alpha + \gamma) \equiv \chi \punt \alpha \dot{+}\chi \punt \gamma, \qquad 
\beta \punt (\alpha \dot{+} \gamma) \equiv \beta \punt \alpha + \beta \punt \gamma.
\label{props}
\end{equation}
We also recall the distributive property of the summation with respect to the dot-product 
$(\alpha + \gamma) \punt \vartheta \equiv \alpha \punt \vartheta + \gamma \punt \vartheta.$ 

{\bf Composition umbra.} The \emph{composition umbra} of $\alpha$ and $\gamma$ is denoted by the symbol 
$\gamma \punt \beta \punt \alpha,$ where $\beta$ is the Bell umbra. Its generating function is
the composition of $f(\alpha,z)$ and $f(\gamma,z),$ that is $f(\gamma \punt \beta \punt \alpha, z) 
= f(\gamma, f(\alpha,z)-1).$ The moments are \cite{Dinardoeurop}
\begin{equation}
E[(\gamma \punt \beta \punt \alpha)^i] = \sum_{j=1}^i E[\gamma^j] B_{i,j}(a_1, \ldots, a_{i-j+1}).
\label{comp}
\end{equation}
As example of composition umbra, the \emph{compositional inverse umbra} 
$\alpha^{\scriptscriptstyle <-1>}$ of an umbra $\alpha$ is such that
$\alpha^{\scriptscriptstyle <-1>} \punt \beta \punt \alpha \equiv \chi \equiv
\alpha \punt \beta \punt \alpha^{\scriptscriptstyle <-1>}.$ In particular, we have
$f(\alpha^{\scriptscriptstyle <-1>},z)=f^{\scriptscriptstyle <-1>}(\alpha,z),$
where $f^{\scriptscriptstyle <-1>}$ denotes the compositional inverse of $f(\alpha,z)$
\cite{Stanley}.
%
\section{L\'evy processes}
Now we focus our attention on the family of auxiliary umbrae $\{t \punt \alpha\}_{t \geq 0}.$
If the moments of $\alpha$ are all finite, this family is the umbral counterpart of a 
stochastic process $\{X_t\}_{t \geq 0}$ such that $E[X_t^k] = E[(t \punt \alpha)^k],$ given in
(\ref{(momlevy)}), for all nonnegative integers $k.$ This stochastic process is a L\'evy process.
\begin{thm} \label{T1}
Let $\{X_t\}_{t \geq 0}$ be a L\'evy process and let $\alpha$ be the umbra such that
$f(\alpha,z) = E[e^{z X_1}].$ Then, the L\'evy process $\{X_t\}_{t \geq 0}$ is umbrally
represented by the family of auxiliary umbrae $\{t \punt \alpha\}_{t \geq 0}.$
\end{thm}
\begin{proof}
Recall that a L\'evy process $\{X_t\}_{t \geq 0}$ is a stochastic process which
starts at $0,$ with independent and stationary increments. If we denote by
$\phi(z,t)$ the moment generating function of the increment $X_{t+s} - X_s$ and by
$\phi(z)$ the moment generating function of $X_1,$ then $\phi(z,t)=(\phi(z))^t,$ due to
the infinite divisibility property \cite{Sato}. The result follows by
observing that we also have $f(t \punt \alpha, z) = [f(\alpha,z)]^t.$
\end{proof}
A fundamental result of the classical umbral calculus is that any umbra is a
partition umbra. In particular, if $\kappa_{{\scriptscriptstyle \alpha}}$ is
the $\alpha$-cumulant umbra, then $\alpha \equiv \beta \punt
\kappa_{{\scriptscriptstyle \alpha}}$ \cite{Dinardoeurop}.
Referring to L\'evy processes, this means that $f(t \punt \alpha,z) =
f(t \punt \beta \punt \kappa_{{\scriptscriptstyle \alpha}},z)
= \exp\{t [f(\kappa_{{\scriptscriptstyle \alpha}},z) - 1 ]\}$
which is very similar to the L\'evy-Khintchine formula \cite{Schoutens},
provided that we specify the expression of $f(\kappa_{{\scriptscriptstyle \alpha}},z).$
Indeed, if we denote by $E[e^{z X_t}] = (\phi(z))^t$ the moment generating function of
a L\'evy process $\{X_t\}_{t \geq 0},$ then the L\'evy-Khintchine
formula is
\begin{equation}
\phi(z) = \exp\left\{z \, m + \frac{1}{2}s^2z^2 +
\int_{\Real}\left(e^{zx} - 1 - zx{\textbf{1}}_{\{|x| \leq 1\}}\right){\rm d}(\nu(x))\right\}.
\label{(lkf)}
\end{equation}
The tern $(m, s^2, \nu)$ is called \emph{L\'evy triplet} and
$\nu$ is the \emph{L\'evy measure}. If $\nu$ is a measure admitting all moments and
if we set $c_0= m + \int_{\{|x| \geq 1\}} x \, {\rm d}(\nu(x)),$
then the L\'evy-Khintchine formula (\ref{(lkf)}) becomes
\begin{equation}
\phi(z) = \exp\left\{c_0 z + \frac{1}{2}s^2z^2\right\} \,
\exp\left\{\int_{ \Real}\left(e^{zx} - 1 - zx\right){\rm d}(\nu(x))\right\}. \label{LK2}
\end{equation}

The following theorem gives an umbral version of a L\'evy process, 
according to the L\'evy-Khintchine formula (\ref{LK2}).

\begin{thm} \label{T2}
A L\'evy process $\{X_t\}_{t \geq 0}$ is umbrally represented by the family 
$\{t \punt \beta \punt [c_0 \chi \dot{+} s \delta \dot{+}\gamma]\}_{t \geq 0},$ where
$\gamma$ is the umbra associated to the L\'evy measure, that is
$f(\gamma,z) = 1 + \int_{\Real}\left(e^{zx} - 1 - zx\right){\rm d}(\nu(x)),$
and $\delta$ is an umbra with $f(\delta,z) = 1 + z^2/2.$ 
\end{thm}

\begin{proof}
We have $f(t \punt \beta \punt [c_0 \chi \dot{+} s \delta \dot{+}\gamma], z) 
= \exp\{t[f(c_0 \chi \dot{+} s \delta \dot{+}\gamma, z) - 1]\}.$
Since $f(c_0 \chi \dot{+} s \delta \dot{+}\gamma, z)
= f(c_0 \chi, z) + f(s \delta, z) + f(\gamma,z) - 2,$ where $f(c_0 \chi, z) = 1 + c_0 z,$ we have
$f(c_0 \chi \dot{+} s \delta \dot{+}\gamma, z) - 1 = \log \phi(z),$ with $\phi(z)$ given in (\ref{LK2}).
\end{proof}

\begin{rem} 
{\rm As introduced in \cite{Dinoli}, the \emph{Gaussian umbra} is the umbra
$m + \beta \punt (s \delta),$ where $m \in \Real,$ $s > 0$ and $\delta$ is
the umbra given in Theorem \ref{T2}. Recalling that $m \equiv \beta \punt \chi \punt m \equiv \beta \punt (m \chi),$
then we have $m + \beta \punt (s \delta) \equiv \beta \punt (m \chi \dot{+} s \delta),$  due to the 
latter of (\ref{props}). Thanks to Theorem \ref{T2} and the latter of (\ref{props}), a L\'evy process 
$\{X_t\}_{t \geq 0}$ is umbrally represented by the family $t \punt \beta \punt [c_0 \chi \dot{+} s \delta \dot{+}\gamma] 
\equiv t \punt \beta \punt [c_0 \chi \dot{+} s \delta] + t \punt \beta \punt \gamma.$
By recalling that the auxiliary umbra $t \punt \beta \punt \alpha$ is the umbral counterpart of
a compound Poisson process $S_N = Y_1 + \cdots + Y_N,$ with $\{Y_i\}$ independent and identically distributed
random variables and $N$ a Poisson random variable of parameter $t,$ then a L\'evy process is the summation
of two compound Poisson processes: in the first, the random variables $Y_i$ are  \textcolor{red}{Gaussian with $c_0$  mean} and variance $s^2,$ in the second the random variables $Y_i$ correspond to the umbra $\gamma$
associated to the L\'evy measure.}
\end{rem}
A centered L\'evy process is such that $E[X_t]=0$ for all $t \geq 0.$ This is equivalent to choose $c_0=0$ 
in equivalence (\ref{LK2}). 
\begin{cor}
A centered L\'evy process is umbrally represented by $\{t \punt \beta 
\punt (s \delta \dot{+}\gamma)\}_{t \geq 0}.$
\end{cor}
The L\'evy process corresponding to (\ref{LK2}) is a martingale if and only if $c_0 = 0,$ see Theorem 5.2.1 in \cite{Applebaum}. 
This means that the singleton umbra $\chi$ plays a central role in the martingale property  
of a L\'evy process. Indeed, if $c_0=0,$ no contribution is given by the singleton umbra which indeed does not admit a probabilistic counterpart.
\subsection{The Kailath-Segall formula}
Let $\{X_t\}_{t \geq 0}$ be a centered L\'evy process with moments of all orders and let
$\{X_t^{(n)}\}_{t \geq 0}$ be the variations (\ref{prima}) of the process. The iterated stochastic integrals
(\ref{seconda}) are related to the variations $\{X_t^{(n)}\}_{t \geq 0}$ by the Kailath-Segall formula \cite{KS}
\begin{equation}
P_t^{(n)} = \frac{1}{n} \left( P_t^{(n-1)} X_t^{(1)} - P_t^{(n-2)} X_t^{(2)} + \cdots +
(-1)^{n+1} P_t^{(0)} X_t^{(n)} \right).
\label{(KS)}
\end{equation}
Then, $P_t^{(n)} = P_n\left(X_t^{(1)}, \ldots, X_t^{(n)}\right)$ is a polynomial in 
$X_t^{(1)}, X_t^{(2)}, \ldots, X_t^{(n)},$ called the $n$-th \emph{Kailath-Segall polynomial}. Let us 
introduce the family of umbrae $\{\Upsilon_t\}_{t \geq 0}$ such that $E[\Upsilon_t^n]= n! 
E\left[P_t^{(n)}\right]$ and $\{\sigma_t\}_{t \geq 0}$ such that $E[\sigma_t^n] = 
E[X_t^{(n)}],$ for all nonnegative integers $n.$ The following theorem states the umbral version of the Kailath-Segall 
formula and its inversion. 

\begin{thm} \label{Inv}
We have $\Upsilon_t \equiv \beta \punt [(\chi \punt \chi) \sigma_t]$  and $(\chi \punt \chi) \sigma_t \equiv \chi \punt \Upsilon_t.$ 
\end{thm}

\begin{proof}
Assume $\psi_t \equiv (\chi \punt \chi) \sigma_t$ where $E[(\chi \punt \chi)^n]=(-1)^{n-1} (n-1)!$ \cite{Dinardoeurop}. 
The recurrence relation (\ref{(KS)}) is equivalent to $E[\Upsilon_t^n] = E[ \psi_t (\Upsilon_t + \psi_t)^{n-1}],$ for all 
$n \geq 1.$ Indeed, by definition of umbrae $\psi_t$ and $\Upsilon_t,$ we have
\begin{align*}
E[\Upsilon_t^n] &  = n ! \, \frac{1}{n}\left\{ \frac{E\left[\Upsilon_t^{n - 1}\right] 
E\left[\psi_t\right]}{(n - 1)!} + \frac{E\left[\Upsilon_t^{n - 2}\right] E\left[\psi_t^2\right]}{(n - 2)!} 
+ \cdots + \frac{E\left[\psi_t^n\right]}{(n - 1)!} \right\} \\
& = \sum_{j = 0}^{n - 1} \binom{n - 1}{j}E\left[\Upsilon_t^{n - 1 - j}\right] 
E\left[\psi_t^{j+1}\right] = E\left[\psi_t (\Upsilon_t + \psi_t)^{n - 1}\right].
\end{align*}
By using the first equivalence of Theorem 3.1 in \cite{Dinoli}, we have $\psi_t \equiv (\chi \punt \chi) \sigma_t 
\equiv \chi \punt \Upsilon_t$ (inversion of the Kailath-Segall formula).
The second equivalence follows by observing that $\psi_t \equiv \chi \punt \Upsilon_t \Leftrightarrow \beta \punt \psi_t 
\equiv \beta \punt \chi \punt \Upsilon_t$ and $\beta \punt \chi \equiv u.$  
\end{proof}
By recalling that the moments of $\beta \punt \alpha$ are the (exponential) complete Bell   
polynomials \cite{Comtet} in the moments of $\alpha$, see \cite{Dinsen} formula (29), then the \emph{Kailath-Segall polynomials} 
are complete Bell exponential polynomials in $\{(-1)^{n-1} (n-1)! E[X_t^{(n)}]\}.$ From the inversion 
of the Kailath-Segall formula and equivalence (\ref{comp}), the following corollary follows. 
\begin{cor}
If $c_i=i! E\left[P_t^{(i)}\right]$ for $i=1, \ldots, n,$ then
$$E[X_t^{(n)}] = \sum_{j=1}^n \frac{(-1)^{n-j}}{(n-1)_{n-j}} B_{n,j} (c_1, c_2, \ldots, c_{n-j+1}).$$ 
\end{cor}
The inversion of the Kailath-Segall formula in Theorem \ref{Inv} is a generalization of formula (3.2)
in \cite{Bernoulli} which gives the elementary symmetric polynomials in terms of
power sum symmetric polynomials. That is, if we replace the jumps $\{\Delta X_s\}$ in $X_t^{(n)}$ with suitable 
indeterminates $\{x_s\},$ then the Kailath-Segall polynomials reduce to the polynomials given in \cite{Taqqu}.

\subsection{Umbral time-space harmonic polynomials}
Let us recall the definition of conditional evaluation given in \cite{{Dinoli2}}. Denote by ${\mathcal X}$ the set 
${\mathcal X} = \{\alpha\}.$
\begin{defn}\label{condeval1}
The linear operator $E(\;\cdot\; \vline \,\, \alpha):\, {\mathbb R}[x][\A]
\;\longrightarrow\; {\mathbb R}[{\mathcal X}]$
such that
\begin{itemize}
\item[{\it i)}] $E(1 \,\, \vline \,\,\alpha)=1$;
\item[{\it ii)}] $E(x^m \alpha^n \gamma^i \delta^j\cdots \, \,  \vline \,\, \alpha)=x^m \alpha^n
 E[\gamma^i]E[\delta^j]\cdots$ for uncorrelated umbrae
$\alpha, \gamma, \delta, \ldots$ and for nonnegative integers $m,n,i,j,\ldots$
\end{itemize}
is called \emph{conditional evaluation} with respect to $\alpha.$
\end{defn}

In other words, Definition \ref{condeval1} says that the conditional evaluation
with respect to $\alpha$ handles the umbra $\alpha$ as it
was an indeterminate.

\begin{defn}\label{condeval}
Let $\{P(x,t)\} \in {\mathbb R}[x]$ be a family of polynomials indexed by $t \geq 0.$
$P(x,t)$ is said to be a {\rm time-space harmonic polynomial} with respect to the family
of auxiliary umbrae $\{q(t)\}_{t \geq 0}$ if and only if
$E \left[ P(q(t),t) \, \, \vline \, \, q(s) \right] = P(q(s),s)$ for all $0 \leq s \leq t.$
\end{defn}

\begin{thm}\label{UTSH2}
The family of polynomials $\{Q_k(x,t)\}_{t \geq 0} \in \Real[x],$ where $Q_k(x,t)=E[(x-t \punt \alpha)^k]$ 
for all nonnegative integers $k,$ is time-space harmonic with respect to $\{t \punt \alpha\}_{t \geq 0}.$
\end{thm}
The proof of Theorem \ref{UTSH2} is in \cite{Dinoli2}.  
\begin{rem}\label{propcomb}
{\rm Every linear combination of $\{Q_k(x,t)\}_{t \geq 0}$ is a time-space harmonic polynomial 
with respect to $\{t \punt \alpha\}_{t \geq 0}.$}
\end{rem}

Theorem \ref{UTSH2} guarantees that the polynomial
\begin{equation}
Q_k(x,t) = E [ ( x - t \punt \beta \punt [c_0 \chi \dot{+} s \delta \dot{+}\gamma]
)^k] 
\label{(TSHL)}
\end{equation}
of degree $k$ in the variable $x$ and depending on the parameter $t,$ is time-space harmonic with respect to the 
family of auxiliary umbrae $\{t \punt \beta \punt [c_0 \chi \dot{+} s \delta \dot{+}\gamma]\}_{t \geq 0},$ 
that is with respect to a L\'evy process, thanks to Theorem \ref{T2}. The following theorem generalizes
Corollary 2(a) in \cite{Sole}.

\begin{prop} \label{bellcompl}
We have $Q_k(x,t) = Y_k(x - t c_0,-t(s^2+m_2),-t m_3,\ldots,-t m_k),$
for all nonnegative integers $k \geq 1,$ where $Y_k$ are the complete Bell polynomials  and 
$m_i=E[\gamma^i],$ for all $i = 2, \ldots, k.$
\end{prop}

\begin{proof}
As proved in \cite{Dinardoeurop}, if $a_k = E[\alpha^k]$ and $b_i = E[\kappa_{{\scriptscriptstyle \alpha}}^i],$ 
for $i = 1,\ldots,k$ then $a_k = Y_k(b_1,b_2,\ldots,b_k),$ where $Y_k$ are the complete Bell polynomials.

By definition of cumulant umbra and by virtue of equivalence (\ref{(TSHL)}), we have
\begin{equation*}
E\{[\chi \punt  (x - t \punt \beta \punt [c_0 \chi \dot{+} s \delta \dot{+}\gamma])]^k\} =
\begin{cases} 
x + E\{\chi \punt (- t) \punt \beta \punt [c_0 \chi \dot{+} s \delta \dot{+}\gamma]\}, & \mbox{if } k = 1 \\ 
E\{(\chi \punt (- t) \punt \beta \punt [c_0 \chi \dot{+} s \delta \dot{+}\gamma])^k\}, & \mbox{if } k >1.
\end{cases}
\end{equation*}

Therefore, since $E\{\chi \punt (-t) \punt \beta \punt [c_0 \chi \dot{+} s \delta \dot{+}\gamma]\} = -t c_0,$ 
$E\{(\chi \punt (-t) \punt \beta \punt [c_0 \chi \dot{+} s \delta \dot{+}\gamma])^2\} = -t(s^2 + m_2)$ 
and $E\{(\chi \punt (-t) \punt \beta \punt [c_0 \chi \dot{+} s \delta \dot{+}\gamma])^k\} = -t m_k,$ for $k\geq 3,$ 
the result follows. \textcolor{red}{Indeed we have proved} that the cumulant umbra of $x - t \punt \beta \punt [c_0 \chi \dot{+} 
s \delta \dot{+}\gamma]$ has the first $k$ moments given by $x- t c_0, -t(s^2+m_2), \ldots, -tm_k.$
\end{proof}

We observe that the polynomial umbra $x - t \punt \beta \punt [c_0 \chi \dot{+} s \delta \dot{+}\gamma]$
is an Appell umbra with respect to the indeterminate $x$ \cite{Sheffer}. Therefore the moments 
$\{Q_k(x,t)\}_{k \in \mathbb{N}}$ in (\ref{(TSHL)}) are Appell polynomials such that 
$\partial Q_k(x,t)/\partial x = k Q_{k-1}(x,t).$ With respect to $t,$ the polynomial umbra 
$x - t \punt \beta \punt [c_0 \chi \dot{+} s \delta \dot{+}\gamma]$ is a Sheffer umbra 
\cite{Sheffer}, so that the Sheffer identity holds $Q_k(x,t + v)=\sum_{j=0}^k \binom{k}{j} P_j(v) Q_{k-j}(x,t),$ where 
$Q_{k - j}(x,t)$ are given in (\ref{(TSHL)})  and $P_j(v)=Q_{j}(0,v),$ for all nonnegative integers $j.$
\section{Examples}
\subsection{Sum of two independent L\'evy processes}
Let us consider two independent L\'evy processes $W = \{W_t\}_{t \geq 0}$ and $ Z =\{Z_t\}_{t \geq 0},$ 
umbrally represented by $\{t \punt \alpha\}_{t \geq 0}$ and $\{t \punt \gamma\}_{t \geq 0},$ respectively. 
Due to the distributive property (\ref{(dist)}), the process $X = W + Z$ is umbrally represented by
$t \punt (\alpha + \gamma) \equiv t \punt \alpha + t \punt \gamma.$ If we replace $\Real[x]$ with $\Real[x,w,z]$ 
\cite{Bernoulli}, and denote by $\{Q_k(x,t)\}_{k \in N},$ $\{Q^{\prime}_k(x,t)\}_{k \in N}$ and 
$\{Q^{\prime \prime}_k(x,t)\}_{k \in N}$ the time-space harmonic polynomials with respect to $\{X_t\}_{t \geq 0},$ 
$\{W_t\}_{t \geq 0}$ and $\{Z_t\}_{t \geq 0},$ respectively, we have $Q_k(x,t) = \sum_{j=0}^k {k \choose j} 
Q^{\prime}_j(w,t)Q^{\prime \prime}_{k-j}(z,t),$ if $x = w+z.$ 
\subsection{Brownian motion} 
The Brownian motion $\{B_t\}_{t \geq 0}$ is a L\'evy process whose increments are Gaussian random variables with 
zero mean, variance $s^2$ and zero L\'evy measure. Hence, thanks to Theorem \ref{T2}, the umbral counterpart 
of $\{B_t\}_{t \geq 0}$ is given by the family of umbrae $\{t \punt \beta \punt (s \delta)\}_{t \geq 0}.$ 
The standard Brownian motion is recovered by setting $s = 1.$

From Theorem \ref{UTSH2}, for all nonnegative integers $k,$ the polynomials $Q_k(x,t) = E[(x - t \punt \beta \punt 
(s\delta))^k]$ are time-space harmonic  with respect to the Brownian motion $\{B_t\}_{t \geq 0}$. 

\begin{prop}
For all nonnegative integers $k \geq 1,$ we have $Q_k(x,t) = H_k^{(s^2 t)}(x).$
\end{prop}

\begin{proof}
Recall that \textcolor{red}{the} generalized Hermite polynomials $\{H_k^{(s^2)}(x)\}_{t \geq 0}$ have generating function $$\sum_{k \geq 0} H_k^{(s^2)}(x)\frac{z^k}{k!} = \exp\left\{xz - \frac{s^2 z^2}{2}\right\}.$$

In \cite{Dinoli2} we have proved that $H_k^{(s^2)}(x) = E\{[x - 1 \punt \beta \punt (s\delta)]^k\}.$ 
In particular we have $H_k^{(s^2 t)}(x) = E[(x - 1 \punt \beta \punt (\sqrt{t}s\delta))^k].$ The result follows 
by observing that $-1 \punt \beta \punt (\sqrt{t}s\delta) \equiv -t \punt \beta \punt s\delta.$
\end{proof}
\subsection{Poisson process}
The Poisson process $\{N_t\}_{t \geq 0}$ is a pure jump L\'evy process, whose increments follow a Poisson 
distribution with parameter $\lambda > 0$. The moment generating function is $(\phi(z))^t = (\exp\{\lambda(e^t - 1)\})^t,$ 
so the Poisson process of intensity parameter $\lambda$ is umbrally represented by the family 
of umbrae $\{t \punt \lambda \punt \beta\}_{t \geq 0}.$

Thanks to Theorem \ref{UTSH2}, the polynomials $Q_k(x,\lambda t) = E[(x - t \punt \lambda \punt \beta)^k]$ are 
time-space harmonic with respect to the Poisson process $\{N_t\}_{t \geq 0}$.

Proposition \ref{aaa} states that also the Poisson-Charlier polynomials $\{\widetilde{C}_k(x,\lambda t)\}$ are time-space harmonic 
with respect to the Poisson process $\{N_t\}_{t \geq 0}.$  

\begin{prop} \label{aaa}
We have $\widetilde{C}_k(x,\lambda t) = \sum_{j=1}^k s(k,j) Q_j(x,\lambda t),$
where $s(k,j)$ are the Stirling numbers of the first kind.
\end{prop}

\begin{proof}
Recall that the Poisson-Charlier polynomials $\widetilde{C}_k(x,\lambda t)$ have generating function 
$$\sum_{k\geq 0}\widetilde{C}_k(x,\lambda t)\frac{z^k}{k!}=e^{-\lambda tz}(1+z)^x,$$ so  
$\widetilde{C}_k(x,\lambda t) = E[(x \punt \chi - t \punt \lambda \punt u)^k].$
Since $x \punt \chi - t \punt \lambda \punt u \equiv (x - t \punt \lambda \punt \beta) \punt \chi$ and by recalling that 
$E[(\alpha \punt \chi)^k]=E[(\alpha)_k],$ see \cite{Dinardoeurop}, we have
\begin{align*}
\widetilde{C}_k(x,\lambda t) = E[(x - t \punt \lambda \punt \beta)_k] = \sum_{j = 0}^k s(k,j) E[(x 
- t \punt \lambda \punt \beta)^k].
\end{align*}
\end{proof}
\subsection{Gamma process}
The Gamma process $\{G_t(\lambda, b)\}_{t \geq 0}$ with scale parameter $\lambda > 0$ and shape parameter 
$b > 0$ is a L\'evy process with stationary, independent and Gamma-distributed increments. If we set $b=1,$ 
the moment generating function of the Gamma process is $(\phi(z))^t = [(1 - z)^{-\lambda}]^t.$ Thus the 
umbral representation of the Gamma process $\{G_t(\lambda, 1)\}_{t \geq 0}$ is given by the family of umbrae 
$\{(\lambda t) \punt \bar{u}\}_{t \geq 0},$ where $\bar{u}$ is the \emph{boolean unity}.

There are two families of polynomials time-space harmonic with respect to Gamma processes, according to the value of the scale parameter 
$\lambda$: the \emph{Laguerre polynomials} $\{\mathcal{L}_k^{t-k}(x)\}$ and the \emph{actuarial polynomials} 
$\{g_k(x, \lambda t)\}.$

As regards the former, we have  
\begin{equation}
(-1)^k k! \mathcal{L}_k^{t-k}(x) = E[(x + t \punt (-\chi))^k], \qquad k=0,1,2,\ldots
\label{laguerre}
\end{equation}
since the Laguerre polynomials $\{\mathcal{L}_k^{t-k}(x)\}$ have generating 
function $\sum_{k\geq 0}(-1)^k \mathfrak{L}_k^{t-k}(x)z^k = (1-z)^t e^{zx}.$ 

\begin{thm}
The Laguerre polynomials $\{\mathcal{L}_k^{t-k}(x)\}_{t \geq 0}$ are time-space harmonic with respect to the 
Gamma process $\{G_t (1, 1)\}_{t \geq 0}.$
\end{thm}

\begin{proof}
Theorem \ref{UTSH2} implies that the polynomials $Q_k(x,t) = E[(x - t \punt \bar{u})^k]$ are time-space harmonic 
with respect to the Gamma process $\{G_t(1,1)\}_{t \geq 0}.$ Moreover, we have $-1 \punt \bar{u} \equiv -\chi,$ so 
$-t \punt \bar{u} \equiv t \punt (-\chi)$ and 
$x - t \punt \bar{u} \equiv x + t \punt (-\chi).$ Then, thanks to (\ref{laguerre}), we have $Q_k(x,t) = k! 
(-1)^k \mathcal{L}_k^{t-k}(x).$
\end{proof}

For the latter, Roman \cite{Roman} defines the class of the actuarial polynomials as the sequence 
of polynomials with generating function
\begin{equation}
\sum_{k \geq 0} g_k(x, \lambda t) \frac{z^k}{k!} = \exp\{\lambda t z + x (1 - e^z)\}.
\label{act_gf}
\end{equation}
To get the umbral expression of $g_k(x, \lambda t)$ we use the umbral L\'evy-Sheffer systems.
Recall that a L\'evy-Sheffer system \cite{Dinoli2} is a sequence of polynomials $\{R_k(x,t)\}$ such that
\begin{equation}
\sum_{k \geq 0} R_k(x,t) \frac{z^k}{k!} = (f(z))^t \exp\{x u(z)\},
\label{LS_gf}
\end{equation}
where
$f(z)$ and $u(z)$ are analytic in the neighborhood of $z = 0,$ $u(0) = 0,$ $f(0) = 1,$ $u'(0) \neq 0$
and $1/f(\tau(z))$ is an infinitely divisible moment generating function, with
$\tau(z)$ such that $\tau(u(z)) = z.$ If $f(z) = f(\alpha,z)$ and $u(z) = f(\gamma,z) - 1,$ then $R_k(x,t) 
= E[(t \punt \alpha + x \punt \beta \punt \gamma)^k],$ for all nonnegative integers $k.$
By comparing (\ref{LS_gf}) with (\ref{act_gf}), we obtain  $\alpha \equiv (\lambda t) \punt u$ and 
$\gamma \equiv (\chi \punt (-\chi))^{{\scriptscriptstyle <-1>}},$ where $(\chi \punt 
(-\chi))^{\scriptscriptstyle <-1>}$ is the compositional inverse of the umbra $\chi \punt (-\chi).$ 
This leads to the umbral version of the actuarial polynomials, that is, for all nonnegative integers $k,$
\begin{equation}
g_k(x, \lambda t) = E\left\{\left[ \lambda t + x \punt \beta \punt 
(\chi \punt (-\chi))^{{\scriptscriptstyle <-1>}} \right]^k\right\}.
\label{act}
\end{equation}

\begin{thm}
The actuarial polynomials $\{g_k(x, \lambda t)\}_{t \geq 0}$ are time-space harmonic with respect to the 
Gamma process $\{G_t (\lambda, 1)\}_{t \geq 0}.$
\end{thm}

\begin{proof}
By virtue of Theorem \ref{UTSH2}, $Q_k(x,t) = E[(x - (\lambda t) \punt \bar{u})^k]$ are time-space 
harmonic polynomials for all $k \geq 0$ with respect to the Gamma process $\{G_t (\lambda, 1)\}_{t \geq 0}.$ 
On the other hand, $\lambda t + x \punt \beta \punt (\chi \punt (-\chi))^{{\scriptscriptstyle <-1>}} 
\equiv (x + (\lambda t) \punt (-\chi)) \punt \beta \punt (\chi \punt (-\chi))^{{\scriptscriptstyle <-1>}}.$
Then, by virtue of (\ref{act}) and (\ref{comp}), we have
\begin{align*}
g_k(x, \lambda t) = \sum_{j = 1}^k E[(x + (\lambda t) \punt (-\chi))^j] B_{k,j}(m_1, \ldots, m_{k - j + 1}),
\end{align*}
where $m_i = E[(\chi \punt (-\chi))^{{\scriptscriptstyle <-1>}})^i].$ Observe that $\bar{u} 
\equiv -1 \punt (-\chi),$ thus
\begin{equation*}
g_k(x, \lambda t) = \sum_{j = 1}^k E[(x - (\lambda t) \punt \bar{u})^j] B_{k,j}(m_1, \ldots, m_{k - j + 1}) 
= \sum_{j = 1}^k Q_k(x,t) B_{k,j}(m_1, \ldots, m_{k - j + 1}).
\end{equation*}
The result follows from Remark \ref{propcomb}.
\end{proof}
\subsection{Pascal process} 
Let $\{Pa(t,p)\}_{t \geq 0}$ be a Pascal process, that is, a L\'evy process whose increments have Pascal 
distribution with mean $t d,$ where $d=p/q$ and $p + q = 1.$ As the moment generating function of the 
Pascal process is $(\phi(z))^t = [q/(1 - p e^z)]^t,$ with some calculations we obtain that a Pascal process 
is umbrally represented by the family of umbrae $\{t \punt \bar{u} \punt d \punt \beta\}_{t \geq 0},$ where 
$\bar{u}$ is the boolean unity. By virtue of Theorem  \ref{UTSH2}, the time-space harmonic polynomials with 
respect to the Pascal process are $Q_k(x,t) = E[(x - t \punt \bar{u} \punt d \punt \beta)^k]$ for all 
nonnegative integers $k.$ 

Consider the family of Meixner polynomials of the first kind $\{M_k(x,t,p)\}$ \cite{Schoutens} such that
\begin{equation}
\sum_{k \geq 0} (-1)^k (t)_k M_k(x,t,p) \frac{z^k}{k!} = \left(1 + \frac{z}{p}\right)^x (1 + z)^{- x - t}.
\label{mgf_meixner}
\end{equation}

From (\ref{mgf_meixner}), the umbral expression of the Meixner polynomials of the first kind is
\begin{equation}
(-1)^k (t)_k M_k(x,t,p) = E\left\{\left[x \punt \left(-1 \punt \chi + \frac{\chi}{p}\right) 
- t \punt \chi\right]^k\right\}.
\label{Meixner}
\end{equation}

\begin{thm}
The Meixner polynomials of the first kind are time-space harmonic with respect to the Pascal process 
$\{Pa(p,t)\}_{t \geq 0}.$
\end{thm}

\begin{proof}
The Meixner polynomials of the first kind form a L\'evy-Sheffer system, so they are represented 
by the polynomial umbra $x \punt \beta \punt (\chi \punt (-1 \punt \chi + \chi/p)) + t \punt (-1 \punt \chi),$ 
with $E[(-1 \punt \chi + \chi/p)] \neq 0.$ This hypothesis guarantees that the 
compositional inverse umbra exists, so
\begin{equation*}
x \punt \left(-1 \punt \chi + \frac{\chi}{p}\right) - t \punt \chi \equiv  (x + t \punt (-1 \punt \chi) 
\punt \beta \punt \bar{u} \punt d \punt \beta) \punt \beta \punt \left(\chi \punt \left(-1 \punt \chi 
+ \frac{\chi}{p}\right)\right).
\end{equation*}
Thus, by (\ref{comp}) and (\ref{Meixner}), the Meixner polynomials of first kind can be written in the following way
\begin{align*}
(-1)^k (t)_k M_k(x,t,p) & = \sum_{j = 1}^k E[(x - t \punt \bar{u} \punt d \punt \beta)^j] 
B_{k,j}(m_1, \ldots, m_{k - j -1})\\ 
& = \sum_{j = 1}^k Q_k(x,t) B_{k,j}(m_1, \ldots, m_{k - j -1}),
\end{align*}
where $m_i = E[\{\chi \punt (-1 \punt \chi + \chi/p)\}^i].$ The result follows, thanks to Remark \ref{propcomb}.
\end{proof}
\subsection{Random walks}
The results in the literature involving the polynomials we are going to introduce refer to
an integer parameter $n.$ In order to highlight their time-space harmonic property, we can consider 
the discrete version of a L\'evy process, that is a random walk $S_n = X_1 + X_2+ \cdots 
+ X_n,$ with $\{X_i\}$ independent and identically distributed random variables.
For the symbolic representation
of a L\'evy process we have dealt with, a random walk is umbrally represented by $n \punt \alpha.$ The generality of the 
symbolic approach shows that if the parameter $n$ is replaced by $t,$ that is if the random walk is 
replaced by a L\'evy process, more general classes of polynomials can be recovered for which many of 
the properties here introduced still hold.

\noindent{\sl \textbf{Bernoulli polynomials}.} 
The Bernoulli polynomials $\{B_k(x,n)\}$ are defined by the generating function \cite{Roman}
$$\sum_{k \geq 0} B_k(x,n) \frac{z^k}{k!} = \left( \frac{z}{e^z-1} \right)^n e^{z x}.$$ 
Therefore we have $B_k(x,n) = E[(x + n \punt \iota)^k]$ for all nonnegative integers $k.$
\begin{thm}\label{bernpol}
The Bernoulli polynomials $\{B_k(x,n)\}$ are time-space harmonic with respect to the random walk
$\{n \punt (-1 \punt \iota)\}_{n \geq 0}.$ 
\end{thm}
\begin{proof}
Let us consider the random walk $S_n = X_1 + X_2 + \cdots + X_n$  such that $X_1, X_2, \ldots, X_n$ 
are $n$ independent and identically distributed random variables with uniform distribution on the interval 
$[0,1].$ Since $X_i$ is umbrally represented by the umbra $-1 \punt \iota,$ for all $i = 1, \ldots, n,$ the 
random walk $S_n$ is umbrally represented by the family of auxiliary umbrae $\{n \punt (-1 \punt 
\iota)\}_{n \geq 0}.$ Theorem \ref{UTSH2} ensures that polynomials $Q_k(x,t) = E[(x - n \punt (-1 \punt 
\iota))^k]$ are time-space harmonic with respect to $S_n,$ for all $k \geq 0.$
On the other hand, $n \punt (-1 \punt \iota) \equiv - n \punt \iota,$ hence $E[(x - n \punt (-1 \punt \iota))^k] 
= E[(x + n \punt \iota)^k],$ that is, $B_k^{(n)}(x) = Q_k(x,n).$
\end{proof}
\noindent{\sl \textbf{Euler polynomials}.} 
The Euler polynomials $\{\mathcal{E}_k (x, n)\}$ are defined by the generating function \cite{Roman}
$$\sum_{k \geq 0} \mathcal{E}_k(x,n) \frac{z^k}{k!} = \left( \frac{2}{e^z+1} \right)^n e^{z x}.$$
Therefore we have
$\mathcal{E}_k(x,n) = E[\left(x + n \punt \left[ \frac{1}{2} \left( - 1 \punt u + \eta \right) \right]\right)^k]$ for all nonnegative integers $k.$
\begin{thm}
The Euler polynomials $\{\mathcal{E}_k (x, n)\}$ are time-space harmonic with respect to the random walk 
$\{n \punt \left[ \frac{1}{2} \left( - 1 \punt \eta + u \right) \right]\}_{n \geq 0}.$ 
\end{thm}

\begin{proof}
Let us consider the random walk $S_n = X_1 + X_2 + \cdots + X_n$  such that $X_1, X_2, \ldots, X_n$ are 
$n$ independent and identically distributed Bernoulli random variables with parameter 
$1/2.$ The result follows by using arguments similar to the proof of Theorem \ref{bernpol}, 
as $X_i$ is umbrally represented by the umbra $ \frac{1}{2} \left( - 1 \punt \eta + u \right) .$ 
\end{proof}

\noindent{\sl \textbf{Krawtchouk polynomials}.}  The Krawtchouk polynomials $\{\mathcal{K}_k (x, p, n)\}$ are defined 
by the generating function \cite{Roman}
\begin{equation}
 \sum_{k \geq 0} \binom{n}{k} \mathcal{K}_k (x, p, n) z^k = \left(1 - \frac{1 - p}{p} z\right)^x (1 + z)^{n - x}.
\label{kraw_gf}
\end{equation}

The umbra with generating function (\ref{kraw_gf}) is $(n - x) \punt \chi 
+ x \punt (-\chi/d) \equiv n \punt \chi + x \punt (-1 \punt \chi -\chi/d),$ with $d = p/q$ and $p + q = 1.$ Then, 
for all nonnegative integers $k$ we have
\begin{equation}
\frac{n!}{(n - k)!}\mathcal{K}_k (x, p, n) = E\left\{\left[n \punt \chi + x \punt \left(-1 \punt \chi 
- \frac{\chi}{d}\right)\right]^k\right\}.
\label{kraw}
\end{equation}

\begin{thm}
The Krawtchouk polynomials are time-space harmonic with respect to the random walk $\{n \punt (-1 \punt \mu)\}_{n \geq 0},$ 
where $-1 \punt \mu$ is the umbral counterpart of a Bernoulli random variable with parameter $p.$
\end{thm}

\begin{proof}
For $i=1, \ldots,n,$ let $X_i$ be a random variable with Bernoulli distribution of parameter $p.$ Let 
$\mu \equiv -1 \punt \chi \punt p \punt \beta$ be the umbra such that $f(\mu,z) = 1/(p e^{z} + (1 - p)),$ so the random walk 
$S_n = X_1 + X_2 + \cdots + X_n$ is umbrally represented by the family of auxiliary umbrae $\{n \punt (-1 \punt \mu)\}_{n \geq 0.}$  
From Theorem \ref{UTSH2}, the polynomials $Q_k(x,n)=E[(x - n \punt (-1 \punt \mu))^k] = E[(x + n \punt \mu)^k]$ are 
time-space harmonic with respect to $S_n$ for all nonnegative integers $k.$ 
From (\ref{kraw}), we have
\begin{equation*}
n \punt \chi + x \punt \left(-1 \punt \chi - \frac{\chi}{d}\right) \equiv \left( x + n \punt \left(\chi \punt 
\left(-1 \punt \chi - \frac{\chi}{d}\right)\right)^{\scriptscriptstyle <-1>}\right) \punt \beta \punt \left(\chi 
\punt \left(-1 \punt \chi - \frac{\chi}{d}\right)\right).
\end{equation*}
By applying (\ref{comp}), the Krawtchouk polynomials are
\begin{equation}
\frac{n!}{(n - k)!}\mathcal{K}_k (x, p, n) = \sum_{j = 1}^k E\left[\left\{ x + n \punt \left[\chi \punt 
\left(-1 \punt \chi - \frac{\chi}{d} \right)\right]^{\scriptscriptstyle <-1>}\right\}^j\right] 
B_{k,j}(m_1, \ldots, m_{k - j + 1}),
\label{kraw2}
\end{equation}
where $m_i = E[(\chi \punt (-1 \punt \chi - \chi/d)^i].$ 
Via generating function, it is straightforward to prove that $-1 \punt (\chi \punt (-1 \punt \chi - 
\chi/d ))^{\scriptscriptstyle <-1>} \equiv \mu,$ therefore $E[\{ x + n \punt (\chi \punt (-1 \punt \chi 
- \chi/d))^{\scriptscriptstyle <-1>}\}^j] = E[(x + n \punt \mu)^j] = Q_j(x,n).$ 
By replacing this result in (\ref{kraw2}), we have
\begin{equation*}
\frac{n!}{(n - k)!}\mathcal{K}_k (x, p, n) = \sum_{j = 1}^k Q_j(x,n) B_{k,j}(m_1, \ldots, m_{k - j + 1}),
\end{equation*}
and the result follows, thanks to Remark \ref{propcomb}.
\end{proof}
\noindent{\sl \textbf{Pseudo-Narumi polynomials}.} The family of pseudo-Narumi polynomials 
$\{N_k (x, a n)\},$ $a \in \mathbb{N},$ is the sequence of the coefficients of the following power series \cite{Roman}

\begin{equation}
 \sum_{k \geq 0} N_k (x, a n) z^k = \left(\frac{\log(1 + z)}{z}\right)^{an} (1 + z)^x.
\label{narumi_fg}
\end{equation}

From (\ref{narumi_fg}), the pseudo-Narumi polynomials result to be the moments of the umbra 
$x \punt \chi + (an) \punt u^{\scriptscriptstyle <-1>}_{\scriptscriptstyle P},$ where 
$u^{\scriptscriptstyle <-1>}_{\scriptscriptstyle P}$ is the primitive umbra of the compositional 
inverse $u^{\scriptscriptstyle <-1>}.$ We recall that, given an umbra $\alpha \in \A,$ the 
$\alpha$-\emph{primitive umbra} $\alpha_{\scriptscriptstyle P}$ is such that $f(\alpha_{\scriptscriptstyle P},z) 
= (f(\alpha,z) - 1)/z.$ For all nonnegative integers $k,$ we have 
\begin{equation}
k! N_k (x, a n) = E\{[(an) \punt u^{\scriptscriptstyle <-1>}_{\scriptscriptstyle P} + x \punt \chi]^k\}.
\label{kr}
\end{equation}

\begin{thm}
The pseudo-Narumi polynomials are time-space harmonic with respect to the random walk 
$\{(an). (-1 \punt \iota)\}_{n \geq 0}.$
\end{thm}

\begin{proof}
Consider the random walk $S_n = X_1 + X_2 + \cdots + X_n,$ where, for $i=1, \ldots,n,$ $X_i$ is a sum 
of $a \in \mathbb{N}$ random variables with uniform distribution on the interval $[0,1].$ Therefore, for $i = 1, 
\ldots, n,$ $X_i$ is umbrally represented by $a \punt (-1 \punt \iota)$ and $S_n$ is umbrally represented 
by $\{n \punt a \punt (-1 \punt \iota)\}_{n \geq 0}.$ By applying Theorem \ref{UTSH2}, it is straightforward 
to prove that $Q_k(x,n) = E[(x - (a n) \punt (-1 \punt \iota))^k]$ are time-space harmonic 
with respect to $S_n.$
On the other hand, $x \punt \chi + (an) \punt u^{\scriptscriptstyle <-1>}_{\scriptscriptstyle P} 
\equiv (x + (an) \punt u^{\scriptscriptstyle <-1>}_{\scriptscriptstyle P} \punt \beta) \punt \beta 
\punt u^{\scriptscriptstyle <-1>},$ and then, from (\ref{kr}), 
\begin{align*}
k! N_k (x, a n) = \sum_{j = 1}^k E[(x + (an) \punt u^{\scriptscriptstyle <-1>}_{\scriptscriptstyle P} 
\punt \beta)^j] B_{k,j}(m_1, \ldots, m_{k - j +1}),
\end{align*}
where $m_i = E[(u^{\scriptscriptstyle <-1>})^i].$ To prove the result, it is sufficient to show that 
$E[(x + (an) \punt u^{\scriptscriptstyle <-1>}_{\scriptscriptstyle P} \punt \beta)^j]$ fits 
with the $j$-th time-space harmonic polynomial $Q_j(x,n).$ Via generating function, we have 
$u^{\scriptscriptstyle <-1>}_{\scriptscriptstyle P} \punt \beta \equiv \iota,$ which gives
\begin{equation*}
k! N_k (x, a n) = \sum_{j = 1}^k Q_j(x,n) B_{k,j}(m_1, \ldots, m_{k - j +1}),
\end{equation*}
and the result follows, thanks to Remark \ref{propcomb}.
\end{proof}
%
\section{Acknowledgments}
The authors thank the referees for helpful comments and suggestions. 
%

\end{document}